\documentclass[11pt,a4paper,reqno]{amsart}
\usepackage[english]{babel}
\usepackage[T1]{fontenc}
\usepackage{palatino}
\usepackage{amsmath}
\usepackage{amssymb}
\usepackage{amsthm}
\usepackage{amsfonts}
\usepackage{graphicx}
\usepackage{color}
\usepackage{mathtools}

\usepackage[colorlinks = true, citecolor = black, urlcolor= black]{hyperref}
\newcommand{\N}{\mathbb N}

\renewcommand{\c}{{\mathcal C}}

\newcommand{\R}{\mathbb R}
\newcommand{\T}{\mathbb T}

\newcommand{\Z}{\mathbb Z}

\renewcommand{\H}{\mathbb H}

\renewcommand{\S}{{\mathbb S}}
\newcommand{\Sec}{{\textsection}}

\newcommand{\Aff}{\operatorname{Aff}^{+}(\mathbb{R})}
\newcommand{\AffR}{\operatorname{Aff}^{+}(\mathbb{R})\times \mathbb{R}}

\usepackage{caption}
\DeclareRobustCommand{\SkipTocEntry}[5]{}

\newtheorem{thm}{Theorem}[section]

\newtheorem{prop}[thm]{Proposition}

\newtheorem{defn}[thm]{Definition}

\theoremstyle{definition}
\newtheorem{rem}[thm]{Remark}
\numberwithin{equation}{section}
\author{Katrin F\"assler and Enrico Le Donne}
\address{University of Fribourg, Department of Mathematics, Switzerland}
\email{katrin.faessler@unifr.ch}

\address{University of Jyv\"askyl\"a, Department of Mathematics and Statistics, Finland}
\email{enrico.e.ledonne@jyu.fi}

\thanks{K.F. was partially supported by the Academy of Finland (grant 285159 `\emph{Sub-Riemannian
manifolds from a quasiconformal viewpoint}') and by the  Swiss National Science Foundation (grant 161299 \emph{`Intrinsic rectifiability and mapping theory on the Heisenberg group'}).
\\
E.L.D. was partially supported by the Academy of Finland (grant
288501
`\emph{Geometry of subRiemannian groups}')
and by the European Research Council
 (ERC Starting Grant 713998 GeoMeG `\emph{Geometry of Metric Groups}').
}
\date{\today}

\title[On the classification of 3D Lie groups]{On the quasi-isometric and bi-Lipschitz classification of 3D Riemannian Lie groups}

\begin{document}

\maketitle

\begin{abstract}
This note is concerned with the geometric classification of connected Lie groups of dimension three or less, endowed with left-inva{\-}riant Riemannian metrics. On the one hand, assembling results from the literature, we give a review of the complete classification of such groups up to quasi-isometries and we compare the quasi-isometric classification with the bi-Lipschitz classification. On the other hand, we study the problem wheth{\-}er two quasi-isometrically equivalent Lie groups may be made isometric if equipped with suitable left-invariant Riemannian metrics. We show that this is the case for three-dimensional simply connected groups, but it is not true in general for multiply connected groups. The counterexample also demonstrates that `may be made isometric' is not a transitive relation.
%On the one hand, assembling results from the literature, we compile a complete classification of such groups up to quasi-isometries and bi-Lipschitz homeomorphisms. On the other hand, we study the problem wheth{\-}er two quasi-isometrically equivalent Lie groups may be made isometric if equipped
% with suitable left-invariant Riemannian metrics. We show that this is the case for three-dimensional simply connected groups, but it is not true in general for multiply connected groups. The counterexample also demonstrates that `may be made isometric' is not a transitive relation.
\end{abstract}

\tableofcontents

\newpage
\section{Introduction}

\subsection{List of groups of dimension at most three}

Following the Bianchi classification (see e.g.\ Theorem 1.4 and Theorem 1.5 in \cite[Chapter 7]{gorbatsevich1994lie}), we start by listing the connected real Lie groups of dimension at most three:

\medskip

\textbf{Lie groups of dimension one:}
$\mathbb{R}$,
$\mathbb{T}^1$.

\medskip

\textbf{Lie groups of dimension two:}
 $\mathbb{R}^2$,
$\mathbb{R}\times \mathbb{T}^1$,
$\mathbb{T}^2$,
$\Aff$.

\medskip

\textbf{Lie groups of dimension three:}
$\mathbb{R}^3$,
$\mathbb{R}^2 \times \mathbb{T}^1$,
$\mathbb{R}\times \mathbb{T}^2$,
$\mathbb{T}^3$,
$\mathrm{N}_3(\mathbb{R})$,
$\mathrm{N}_3^{\ast}(\mathbb{R})$,
$\mathrm{SU}(2)$,
$\mathrm{SO}(3)$,
$\widetilde{\mathrm{SL}}(2)$,
$\{\mathrm{PSL}(2)_k\colon\; k\in \N\}$,
$\widetilde{\mathrm{SE}}(2)$,
$\{\mathrm{SE}(2)_k\colon\; k\in \N\}$,
$\mathrm{J}$,
$\{\mathrm{D}_{\lambda}\colon\;0<|\lambda|\leq 1\}$,
$\{\mathrm{C}_{\lambda}\colon\;\lambda >0\}$,
$\AffR$,
$\Aff\times \mathbb{T}^1$.

\medskip

Many of these groups are well known: the $k$-dimensional Euclidean group $\R^k$, the $k$-dimensional torus
$\mathbb{T}^k=(\mathbb{R}/\mathbb{Z})^k$  and direct products of these groups. Nilpotent but non-Abelian groups are the Heisenberg group $\mathrm{N}_3(\mathbb{R})$ and its  quotient $\mathrm{N}_3^{\ast}(\mathbb{R})$ modulo the group of integer points in the center, when $\mathrm{N}_3(\mathbb{R})$ is seen as upper triangular matrix group. Among the solvable but not nilpotent groups there are $\Aff$ (the group of orientation-preserving affine maps of the real line) and products thereof with $\R$ and $\T^1$, moreover
$\widetilde{\mathrm{SE}}(2)$ (the universal cover of the group $\mathrm{SE}(2)$ of orientation preserving isometries of the Euclidean plane) and
$\mathrm{SE}(2)_k$ (the $k$-fold cover of $\mathrm{SE}(2)$). Well-known simple groups are $\mathrm{SU}(2)$ (the special unitary group), $\mathrm{SO}(3)$ (the special orthogonal group), $\widetilde{\mathrm{SL}}(2)$ (the universal cover of     the special linear group), and  $\mathrm{PSL}(2)_k$ (the $k$-fold cover of the projective special linear group $\mathrm{PSL}(2)$).

Apart from  $\widetilde{\mathrm{SL}}(2)$ and $\mathrm{SU}(2)$, all the simply connected groups listed in the previous paragraph
are isomorphic to semidirect products $\mathbb{R}^2 \rtimes_A \mathbb{R}$, where $\mathbb{R}$ acts on $\mathbb{R}^2$ by a matrix  $A\in \mathrm{Mat}(2\times 2,\mathbb{R})$ such that the Lie group product is given by the following expression:
\begin{equation}\label{eq:group_law}
(x,y,z)\ast_A (x',y',z'):=\left(\begin{pmatrix}x\\ y\end{pmatrix}+ e^{zA}\begin{pmatrix}x'\\y'\end{pmatrix},z+z'\right).
\end{equation}
One can find a basis $\{E_1,E_2,E_3\}$ for the Lie algebra of $\mathbb{R}^2 \rtimes_A \mathbb{R}$ whose structure constants are given by
\begin{equation}\label{eq:structure_constants}
A= \begin{pmatrix}c_{13}^1&c_{23}^1\\c_{13}^2&c_{23}^2\end{pmatrix},
\end{equation}
and $c_{ij}^k = 0$ for all other cases where $i\leq j$ and $k\in \{1,2,3\}$, see for instance \cite[\Sec 2.2]{MR2963596}.

The connected $3$-dimensional Lie groups which we have not yet introduced are all solvable and also of the form
$\mathbb{R}^2 \rtimes_A \mathbb{R}$.
For $A=\begin{psmallmatrix}1&1\\0&1\end{psmallmatrix}$ (respectively $\begin{psmallmatrix}1&0\\0&\lambda\end{psmallmatrix}$, respectively $\begin{psmallmatrix}\lambda & 1\\-1& \lambda\end{psmallmatrix}$), we obtain $J$ (respectively $D_{\lambda}$, respectively $C_{\lambda}$).

     \subsection{Classification results}\label{s:relations}

     \textsc{Standing assumption.} \emph{All distances considered are left-invariant Riemannian distances.}
 \medskip

 A (not necessarily continuous) map $\Psi:(X,d) \to (X',d')$ between metric spaces is a \emph{quasi-isometry} if there exist constants $0<C<\infty$ and $1\leq L<\infty$ such that
 \begin{enumerate}
 \item[(i)] $L^{-1}d(x,y)-C\leq d'(\Psi(x),\Psi(y))\leq L d(x,y) + C$ for all $x,y\in X$,
 \item[(ii)] for all $x'\in X'$ there is $x\in  X $ such that $d'(\Psi(x), {x'})\leq C$.
 \end{enumerate}
 If (i) and (ii) hold with $C=0$, the map $\Psi$ is said to be \emph{bi-Lipschitz}, and if moreover $L=1$, then $\Psi$ is an \emph{isometry}. If  $X$ and $X'$ are manifolds and if the distances $d$ and $d'$ are induced by Riemannian metrics $g$ and $g$, respectively, then according to a well-known result by S.\ B.\ Myers and N.\ E.\ Steenrod \cite{MR1503467}, the map $\Psi$ is an isometry exactly if it is a diffeomorphism such that $\Psi^{\ast} g' =g$, see also \cite[Theorem 5.6.15]{MR3469435}.
Since any two left-invariant Riemannian distances on a Lie group are bi-Lipschitz equivalent, we can discuss the quasi-isometric and bi-Lipschitz classification of such groups without specifying a metric. On the other hand, the existence of \emph{isometries} between two groups depends on the choice of metrics. As we are interested in the geometric classification of groups, rather than the classification of groups endowed with a specific metric, we study the following property.

\begin{defn}\label{d:MayBeMadeIsometric} We say that two connected Lie groups $G$ and $H$ \emph{may be made isometric} if there exist left-invariant Riemannian distances $d_G$ and $d_H$ on $G$ and $H$, respectively, such that $(G,d_G)$ and $(H,d_H)$ are isometric.
\end{defn}

Definition \ref{d:MayBeMadeIsometric}  goes back to \cite[\Sec 1.2]{2017arXiv170509648C}, but differs slightly from the original definition, which was formulated for arbitrary topological groups and which required only the existence of  left-invariant distances that induce the manifold topology. By  \cite[Proposition 2.4]{KLD} isometries between connected Lie groups endowed with such distances are actually isometries for some left-invariant \emph{Riemannian} distances, and hence Definition \ref{d:MayBeMadeIsometric} agrees with the definition of \cite{2017arXiv170509648C} in the case of connected Lie groups.

It is easy to show that two Lie groups $G$ and $H$ may be made isometric if and only if there exists a Riemannian manifold $M$ on which both $G$ and $H$ act simply transitively by isometries, see Proposition \ref{p:common_model}.

If $X$ is a fixed model space with a standard distance $d_X$, for instance Euclidean space or the hyperbolic plane, we will also say that ``$G$ may be made isometric to $X$'' if there exists a  left-invariant Riemannian distance $d_G$ on $G$ such that $(G,d_G)$ and $(X,d_X)$ are isometric.

     In Section \ref{s:relations}, we discuss relations between connected Lie groups of dimension at most three in descending order of strength, that is, we list pairs consisting of groups that
     \begin{enumerate}
     \item[(a)] may be made isometric (Proposition \ref{p:I})
     \item[(b)] are bi-Lipschitz  (Proposition \ref{p:BL})
     \item[(c)] are quasi-isometrically homeomorphic (Proposition \ref{p:QIH})
     \item[(d)] are quasi-isometric (Proposition \ref{p:QI}).
     \end{enumerate}
     To conclude the quasi-isometric classification given in Theorem \ref{t:QIClass} below,
     we show that the pairs not appearing in the list (a)--(d) consist of groups that are not quasi-isometrically equivalent.

     Classification problems for Lie groups have a long history that dates back to L.\ Bianchi's isomorphic classification of $3$-dimensional Lie algebras \cite{bianchi1897sugli}. This note is concerned with the geometric classification of Lie groups that are additionally equipped with left-invariant Riemannian distances. Gromov \cite{Gro} in his address to the ICM in 1983 promoted a program to study finitely generated groups with word metrics up to quasi-isometries. This classification problem is related to the quasi-isometric classification of Riemannian manifolds, as the fundamental group of a compact connected Riemannian manifold $M$ is a finitely generated group quasi-isometrically equivalent to the universal Riemannian cover $\widetilde{M}$ according to the \v{S}varc-Milnor lemma.
     % \cite{epstein1992word}

     In the first part of this note, we recall the quasi-isometric classification of connected  Lie groups up to dimension three.
     This is the work of several authors who have studied various aspects of the quasi-isometric classification, for instance for solvable groups of a specific form, or under curvature constraints. We list some of these results: Y.\ Guivarc'h and J.\ W.\ Jenkins' characterization of connected Lie groups with polynomial growth  \cite{MR0369608,MR0316625}, E.\ Heintze's work on solvable Lie groups and homogeneous manifolds of negative curvature \cite{MR0353210}, J.\ Milnor's study of the curvature properties of left-invariant Riemannian metrics on Lie groups \cite{MILNOR1976293}, the study of $3$-dimensional model geometries and Dehn functions in the work of Epstein et al.\ \cite{epstein1992word}  on automatic group,  P.\ Pansu's work on $L^p$ cohomology \cite{Pansu,MR2390047}, Y.\ de Cornulier's computation of the covering dimension of asymptotic cones of connected Lie groups \cite{MR2399134},  the study of quasi-isometries of certain solvable Lie groups by A.\ Eskin, D.\ Fisher, K.\ Whyte \cite{MR2925383}, and X.\ Xie's quasi-isometric classification of negatively curved solvable Lie groups of the form $\R^n \rtimes \R$ \cite{MR3180486}. Depending on the case to be treated, different tools are used in the classification problem, such as  volume growth, Dehn functions, curvature and asymptotic cones of Riemannian manifolds.

\begin{thm}[Various authors]\label{t:QIClass}
All connected real Lie groups of dimension at most three can be classified up to quasi-isometries according to the following table:
\begin{center}
\begin{minipage}[b]{\linewidth}
\centering
\begin{tabular}{|l|l|}
  \hline
  Class & Representatives \\\hline\hline
  $(1)$ & $\mathbb{T}^1$, $\mathbb{T}^2$, $\mathbb{T}^3$, $\mathrm{SU}(2)$, $\mathrm{SO}(3)$ \\
  \hline
  $(2)$ & $\mathbb{R}$, $\mathbb{R}\times \mathbb{T}^1$, $\mathbb{R}\times \mathbb{T}^2$\\
  \hline
  $(3)$& $\mathbb{R}^2$, $\mathbb{R}^2 \times \mathbb{T}^1$, $\mathrm{N}_3^{\ast}(\mathbb{R})$, $\{\mathrm{SE}(2)_k\colon  k\in\N\}$\\
  \hline
  $(4)$& $\mathbb{R}^3$, $\widetilde{\mathrm{SE}}(2)$\\
  \hline
  $(5)$ & $\mathrm{N}_3(\mathbb{R})$\\
  \hline
  $(6)$ & $\widetilde{\mathrm{SL}}(2)$, $\AffR$ \\
  \hline
  $(7_{\lambda})$ \textup{for} $\lambda \in [-1,0)$& $\mathrm{D}_{\lambda}$\\
  \hline
  $(8)$ & $\Aff$, $\Aff\times \mathbb{T}^1$, $\{\mathrm{PSL}(2)_k\colon k\in \N\}$\\
  \hline
  $(9)$& $\mathrm{J}$\\
  \hline
  $(10)$& $\mathrm{D}_1$, $\{\mathrm{C}_{\lambda}\colon \lambda >0\}$\\
  \hline
  $(11_{\lambda})$ \textup{for} $\lambda \in (0,1)$ & $\mathrm{D}_{\lambda}$\\
  \hline
\end{tabular}
%\captionof{table}{Quasi-isometry classes of Lie groups of dimension $\leq 3$}
\label{t:QIclass}
\end{minipage}
\end{center}
\end{thm}

We stress that the classes $(7_{\lambda})$ are distinct for different values of $\lambda$, and the same holds for $(11_{\lambda})$. In Section \ref{s:conclusion}
we will explain how the above mentioned results by various authors can be combined to prove Theorem \ref{t:QIClass}.

\medskip

     According to Theorem \ref{t:QIClass}, two simply connected $3$-dimensional Lie groups $G$ and $H$ (that are not isomorphic) are quasi-isometric to each other if and only if one of the following holds:
     \begin{enumerate}
     \item\label{i:isom} $G,H\in \{\R^3,\widetilde{\mathrm{SE}}(2)\}$
     \item\label{ii:isom} $G,H \in \{\widetilde{\mathrm{SL}}(2),\AffR\}$
     \item\label{iii:isom} $G,H \in \{D_1\}\cup \{C_{\lambda}:\; \lambda >0\}$.
     \end{enumerate}
     In Proposition \ref{p:I} we shall show that in all these cases, the two groups $G$ and $H$ may in fact be made isometric. By Proposition \ref{p:common_model}, this means that there exists a Riemannian manifold $M$ on which both $G$ and $H$ act simply transitively by isometries. In fact, $M$ may be taken equal to a Riemannian manifold that corresponds to one of the eight $3$-dimensional model geometries by Thurston \cite{thurston1997three}:
     \begin{itemize}
     \item the Euclidean geometry in \eqref{i:isom},
     \item the geometry of $\widetilde{\mathrm{SL}}(2)$ in \eqref{ii:isom},
     \item the hyperbolic geometry in \eqref{iii:isom},
     \end{itemize}
      see the discussion in Section \ref{s:I}, and in particular Remark \ref{r:model_SL(2)} for \eqref{ii:isom}. Thus we obtain the following result:

     \begin{thm}\label{t:simplyConn}
     If two non-isomorphic simply connected $3$-dimensional Lie groups are quasi-isometric, then they may be made isometric to one of the eight Thurston geometries.
     \end{thm}

  In Proposition \ref{p:BL} we shall show that without the assumption ``simply connected'', it is not true in general that two connected, quasi-isometric Lie groups may be made isometric. Moreover, since the groups $\mathrm{PSL}(2)_k$, for different values of $k\in \N$, may all be made isometric to  $\Aff \times \T^1$, but cannot be made isometric to each other, we have  the following consequence.

     \begin{prop} The relation ``may be made isometric'' is not transitive.
     \end{prop}

\noindent\textbf{Acknowledgements.} We are grateful to Yves de Cornulier for numerous helpful comments, additions, and suggestions.
In particular we thank him for encouraging us to discuss geometric models, for bringing the reference \cite{MR617247} to our attention and for explaining how it can be used to show that the groups $\mathrm{PSL}(2)_k$ for different values of $k$ cannot be made isometric. We also wish to thank Bruce Kleiner and Ville Kivioja for useful discussions.

  \section{Relations between groups}\label{s:relations}

     \subsection{Groups that may be made isometric}\label{s:I}

     We begin the section with a basic observation about Lie groups that may be made isometric and carry on with a list of $3$-dimensional Lie groups that may be made isometric.

\begin{prop}\label{p:common_model}
Two Lie groups $G$ and $H$ may be made isometric if and only if there exists a Riemannian manifold $M$ on which both $G$ and $H$ act simply transitively by isometries.
\end{prop}

\begin{proof}
Assume first that $G$ and $H$ possess Riemannian distances $d_G$ and $d_H$, respectively, for which there exists an isometry $\Psi: (G,d_G) \to (H,d_H)$. Take $M=H$ equipped with the Riemannian metric $g$ that induces $d_H$. Clearly, $H$ acts on $M$ simply transitively by isometries, and the same is true for $G$ with the action given by
\begin{displaymath}
G\times M \to M,\quad (g,m) \mapsto \Psi \circ L_g \circ \Psi^{-1}(m),
\end{displaymath}
where $L_g$ denotes left translation by $g\in G$.

Conversely, assume that $G$ and $H$ act simply transitively on a manifold $M$ with Riemannian distance $d$. Fix $x_0 \in M$ and define
\begin{displaymath}
d_G(g,g'):= d(g.x_0,g'.x_0),\quad g,g'\in G
\end{displaymath}
and
\begin{displaymath}
d_H(h,h'):= d(h.x_0,h'.x_0),\quad h,h'\in H.
\end{displaymath}
Since by assumption the actions of $G$ and $H$ on $M$ are free, the above definition yields distance functions on $G$ and $H$. From the compatibility of group actions and the fact that $G$ and $H$ act by isometries, we easily deduce that $d_G$ and $d_H$ are left-invariant. For instance, for $G$, we find for
\begin{align*}
d_G(g_0 g, g_0 g')&= d(g_0.(gx_0),g_0 . (g'x_0))
= d(g.x_0,g'.x_0)
= d_G(g,g').
\end{align*}
Since the given actions by $G$ and $H$ on $M$ are also transitive, for every $g\in G$ we find $h(g)\in H$ such that $g.x_0 = h(g).x_0$. This defines a map $(G,d_G)\to (H,d_H)$, $g\mapsto h(g)$, which is easily seen to be an isometry.
\end{proof}

     \begin{prop}\label{p:I}
     Each of the following pairs consists of groups that may be made isometric:
     \begin{enumerate}
     \item\label{I1} $(\R^3,\widetilde{\mathrm{SE}}(2))$
     \item\label{I2} $(\R^2 \times \T^1,\mathrm{SE}(2)_k)$ for every $k\in \N$
     \item\label{I3} $(\mathrm{SE}(2)_k,\mathrm{SE}(2)_{k'})$ for  all $k,k'\in \N$
     \item\label{I4} $(\widetilde{\mathrm{SL}}(2),\AffR)$
     \item\label{I5} $(\Aff\times \T^1,\mathrm{PSL}(2)_k)$ for every $k\in \N$
     \item\label{I6} $(D_1,C_{\lambda})$ for every $\lambda>0$
     \item\label{I7} $(C_{\lambda},C_{\lambda'})$ for all $\lambda,\lambda'>0$.
     \end{enumerate}
     \end{prop}

   \begin{proof}
  It is well known that $\R^3$ and $\widetilde{\mathrm{SE}}(2)$ may be made isometric, see for instance
\cite[Corollary 4.8]{MILNOR1976293}, \cite[Theorem 2.14, (1-b)]{MR2963596}, and \cite[\Sec 4]{KLD}; or read the discussion later in this section.
   The statement that $\mathrm{SE}(2)_k$ may be made isometric to $\R^2 \times \T^1$ is Proposition \ref{p:SE(2)_k_R2S1}. As a corollary, the groups $\mathrm{SE}(2)_k$ and $\mathrm{SE}(2)_{k'}$ for arbitrary $k,k'\in \N$ may be made isometric.  Proposition \ref{p:SL(2)AffR} shows that $\widetilde{\mathrm{SL}}(2)$ and $\AffR$ may be made isometric.
By Proposition \ref{p:PSL(2)_kAffS1}, $\Aff \times \T^1$ may be made isometric to  $\mathrm{PSL}(2)_k$ for every value of $k\in \mathbb{N}$.\\
The items  \eqref{I6} and \eqref{I7} in Proposition \ref{p:I} follow by curvature considerations.
On the (simply connected) groups $D_1$ and on $C_{\lambda}$, $\lambda>0$, one can find a left-invariant Riemannian distance with constant negative sectional curvature: for $D_1$, this follows from Special Example 1.7 in Milnor's article \cite{MILNOR1976293}, for $C_{\lambda}$, $\lambda>0$, it is a consequence of \cite[Theorem 4.11]{MILNOR1976293}; see also \cite[Lemma 2.13 and Theorem 2.14, (1-a)]{MR2963596} and \cite[Introduction]{MR3180486}. It is well known that every simply connected and complete Riemannian manifold with negative constant sectional curvature $K$ is isometric to hyperbolic space in the respective dimension with sectional curvature $K$, hence all the groups $D_1$ and $C_{\lambda}$, $\lambda>0$ may be made isometric to hyperbolic $3$-space, and thus also to each other.
\end{proof}

We now provide the details for the results that have been used in the proof of Proposition \ref{p:I} and for which no other reference has been given. The groups to be considered are $\widetilde{\mathrm{SE}}(2)$, $\widetilde{\mathrm{SL}}(2)$, and quotients thereof.
The simply connected Lie group $\widetilde{\mathrm{SE}}(2)$ is isomorphic to $(\mathbb{R}^3,\ast)$, where
\begin{align*}
(x,y,\theta)\ast (x',y',\theta') &=  \left(\begin{pmatrix}x\\y\end{pmatrix}+ \begin{pmatrix}\cos \theta & -\sin \theta \\ \sin \theta& \cos \theta\end{pmatrix}\begin{pmatrix}x'\\y'\end{pmatrix},\theta + \theta'\right)\\
&= (x+ x'\cos \theta -y'\sin \theta , y + x' \sin \theta + y'\cos \theta, \theta + \theta').
\end{align*}
A direct computation shows that the Euclidean distance $d_E$ on $\mathbb{R}^3$ is left-invariant with respect $\ast$, and hence $\R^3$ and $\widetilde{\mathrm{SE}}(2)$ may be made isometric. It is easy to verify that the sets $(N_k,\ast)$, $k\in\mathbb{N}$, given  by
\begin{displaymath}
N_k = \{(0,0,2\pi k m): m\in \Z\},
\end{displaymath}
are exactly the discrete normal subgroups of $\widetilde{\mathrm{SE}}(2)$.
Every $k\in \mathbb{N}$ gives thus rise to a multiply connected Lie group
\begin{displaymath}
\mathrm{SE}(2)_k:= \widetilde{\mathrm{SE}}(2)/ N_k.
\end{displaymath}
The center of $\mathrm{SE}(2)_k$ contains exactly $k$ elements, which shows that $\mathrm{SE}(2)_k$ is not isomorphic to $\mathrm{SE}(2)_l$ for $k\neq l$.
Moreover, $\mathrm{SE}(2)_k$ is isomorphic to $(\mathbb{R}^2 \times (\mathbb{R}/2\pi k \Z),\ast_k)$, where
\begin{displaymath}
(x,y,\theta) \ast_k (x',y',\theta') = (x+ x'\cos \theta -y'\sin \theta , y + x' \sin \theta + y'\cos \theta, \theta + \theta').
\end{displaymath}

\begin{prop}\label{p:SE(2)_k_R2S1}
For every $k\in\mathbb{N}$, the group $\mathrm{SE}(2)_k$ may be made isometric to the standard round cylinder $\R^2 \times \R/\Z$.
\end{prop}

\begin{proof}
We construct  a left-invariant distance on $\mathrm{SE}(2)_k$, by setting
\begin{equation}\label{eq:SE(2)_k}
d_{SE(2)_k}((x,y,\theta),(x',y',\theta')):= \sqrt{\|(x,y)-(x',y')\|^2+( {(2\pi k)}^{-1}d_{\R/ 2\pi k \Z}(\theta,\theta'))^2}
\end{equation}
for $(x,y,\theta)$ and $(x',y',\theta')$ in $\R^2 \times (\R/ 2\pi k \Z)$. Here
\begin{displaymath}
d_{\R/ 2\pi k \Z}(\theta,\theta'):= \min_{m\in \mathbb{Z}} \{|2\pi k m - (\theta-\theta')|\},
\end{displaymath}
Then the map $\Psi: \R^2 \times \R/\Z \to \R^2 \times (\R/ 2\pi k \Z)$ given by
\begin{displaymath}
\Psi(x,y,{\theta}) = (x,y,2 \pi k\theta)
\end{displaymath}
provides an isometry between $\R^2 \times \R/\Z$ and $\mathrm{SE}(2)_k$.

\end{proof}

We now turn our attention to $\widetilde{\mathrm{SL}}(2)$ and its quotients. Since $\widetilde{\mathrm{SL}}(2)$ is a simple Lie group, \cite[Corollary 3.11]{2017arXiv170509648C} is useful:

\begin{thm}[Cowling et al.]\label{t:Cowlingetal}
Let $G$ be a connected semisimple Lie group and let $G=ANK$ be its Iwasawa decomposition.
Write $K$ as $V \times K'$, where $V$ is a vector
group and $K'$ is compact. Then $G$ may be made isometric to the direct
product $AN \times V \times K'$.
\end{thm}

If $K$ is compact, then $G$ may be made isometric to $AN \times K$. A condition which ensures the compactness of $K$ for a given semisimple Lie group is that $G$ has finite center,
see \cite[p.160 in Chapter 4]{gorbatsevich1994lie}. A connected semisimple Lie group that is linear  has finite center, see for instance \cite[Chapter 1, \Sec 5]{gorbatsevich1994lie}.

The Iwasawa decomposition of $\widetilde{\mathrm{SL}}(2)$ is  $ANK$, where $A$ and $N$ are the following matrix groups
\begin{align*}
A= \left\{\begin{pmatrix}e^t & 0 \\ 0& e^{-t}\end{pmatrix}:\; t\in \mathbb{R}\right\},\quad
N= \left\{\begin{pmatrix}1 & x \\ 0& 1\end{pmatrix}:\; x\in \mathbb{R}\right\},
\end{align*}
and $K$ is isomorphic to $\R$. More precisely, the Iwasawa decomposition is given by the diffeomorphism
\begin{displaymath}
\phi: \R^3 \to \widetilde{\mathrm{SL}}(2)
\end{displaymath}
so that $\phi(0,0,0)=\mathrm{I}$ and
\begin{displaymath}
(\pi \circ \phi)(t,x,\theta) = \begin{pmatrix}e^t & 0 \\ 0& e^{-t}\end{pmatrix} \begin{pmatrix}1 & x \\ 0& 1\end{pmatrix} \begin{pmatrix}\cos \theta & \sin \theta \\ -\sin \theta& \cos\theta\end{pmatrix},
\end{displaymath}
where $\pi:\widetilde{\mathrm{SL}}(2) \to \mathrm{SL}(2)$ is the universal covering projection.
Note that $AN$ is isomorphic to the orientation-preserving affine maps of the real line, that is, to $\Aff$.

Theorem \ref{t:Cowlingetal} applied to the Iwasawa decomposition of  $\widetilde{\mathrm{SL}}(2)$ yields the following statement.

\begin{prop}\label{p:SL(2)AffR}
The groups $\widetilde{\mathrm{SL}}(2)$ and $\AffR$ may be made isometric.
\end{prop}

\begin{rem}\label{r:Aff1}   The group $\Aff$ admits a left-invariant metric of constant negative sectional curvature (see for instance \cite[Special Example 1.7]{MILNOR1976293}) and hence, by the same reasoning as in the proof of Proposition \ref{p:I}, it may be made isometric to the hyperbolic plane $\H^2$.
The quasi-isometric, or even bi-Lipschitz, equivalence of $ \H^2\times \R$ and $ \widetilde{\mathrm{SL}}(2)$ was proved earlier by E.\ Rieffel in her PhD thesis  {\cite{rieffel1993groups}}. The idea of the construction is explained in \cite[\Sec 2]{MR1650098}.
To set the stage, we follow \cite[p.462]{doi:10.1112/blms/15.5.401} and observe that the standard Riemannian metric on $\H^2$ induces a natural Riemannian metric on $T\H^2$ in such a way that for every isometry $f:\H^2 \to \H^2$, the differential $df: T\H^2 \to T\H^2$ is an isometry as well. Since the unit tangent bundle $UT\H^2$ is a submanifold of $T\H^2$, it inherits a Riemannian metric from $T\H^2$, and as $UT(\H^2)$ may be identified with  $\mathrm{PSL}(2)$,
%(see for instance \cite[G-6]{apanasov2000conformal} or \cite[Exercise 2.6.6]{thurston1997three}),
this metric lifts to $\widetilde{\mathrm{SL}(2)}$. One can show that the thus obtained Riemannian metric on $\widetilde{\mathrm{SL}(2)}$
is left-invariant, see \cite[p.464]{doi:10.1112/blms/15.5.401}.

To prove the bi-Lipschitz equivalence of  $ \H^2\times \R$ and $ \widetilde{\mathrm{SL}}(2)$, one constructs  a map
\begin{displaymath}
f:UT( {\H^2})\to  {\H^2} \times \S^1,\quad f(v):=(x,\phi(v)),
\end{displaymath}
 as follows:
 first, one fixes  a point $p_0\in  {\H^2}$, then, for $v\in UT_x( {\H^2})$, the vector $\phi(v)\in UT_{p_0}( {\H^2})$ is obtained by parallel transporting $v$ along the geodesic segment $[xp_0]$. One then proves that $f$ is bi-Lipschitz; see \cite[Proposition 3.10]{KNotes}, and \cite[IV.48]{MR1786869} for more details. Since $f$ lifts to a bi-Lipschitz map between universal covers, see Proposition \ref{p:lift}, this reasoning shows that $ \H^2\times \R$ and $ \widetilde{\mathrm{SL}}(2)$ are bi-Lipschitz equivalent, and in particular quasi-isometrically equivalent.

\end{rem}

\begin{rem}\label{r:nonIsostd}
By Proposition \ref{p:SL(2)AffR}, the group $\widetilde{\mathrm{SL}}(2)$ may be made isometric to $\AffR$.  Moreover, according to Remark \ref{r:Aff1}, the group $\AffR$ may be made isometric to $\H^2 \times \R$ with the standard metric. However, this does not imply that $\widetilde{\mathrm{SL}}(2)$ can be made isometric to the standard $\H^2 \times \R$, and indeed this is not the case: An isometry between the $\H^2 \times \R$ and $\widetilde{\mathrm{SL}}(2)$ with a left-invariant distance would induce a free transitive isometric action of $\widetilde{\mathrm{SL}}(2)$ on $\H^2 \times \R$.
Notice that every isometry $f$ of  $\H^2 \times \R$ sends a set of the form $\H^2 \times\{p\}$ to the set     $\H^2 \times\{f(p)\}$, since these sets are the leaves of the foliation integrating the planes of sectional curvature $-1$.
Thus, if $\widetilde{\mathrm{SL}}(2)$ acts by isometry on $\H^2 \times \R$, then the induced action on $\R$ would be by translations, since $\widetilde{\mathrm{SL}}(2)$ is connected. At the same time, the action would have to be trivial since $\widetilde{\mathrm{SL}}(2)$ is simple, so  it could not act transitively on $\H^2 \times \R$. See also \cite[Section 5]{doi:10.1112/blms/15.5.401}.
\end{rem}

Since the groups $\widetilde{\mathrm{SL}}(2)$ and $\AffR$ may be made isometric, one might wonder if there is a ``standard'' Riemannian manifold to which they may both be made isometric. According to Remark \ref{r:nonIsostd}, this manifold cannot be the standard $\H^2 \times \R$, but it turns out that $\widetilde{\mathrm{SL}}(2)$ endowed with the metric corresponding to one of the Thurston geometries has the desired property; see Remark  \ref{r:model_SL(2)} below.

Consider the left-invariant Riemannian metric $g_{\widetilde{\mathrm{SL}}(2)}$ on $X:=\widetilde{\mathrm{SL}}(2)$ that arises from the identification of $\mathrm{PSL}(2)$ with the unit tangent bundle $UT(\H^2)$ as described in  Remark \ref{r:Aff1} and let $G:=\mathrm{Isome}(\widetilde{\mathrm{SL}}(2))$ be the corresponding isometry group. Then $(X,G)$ is one of the eight three-dimensional model geometries of Thurston \cite[Theorem 3.8.4]{thurston1997three}.
Clearly, $\widetilde{\mathrm{SL}}(2)$ acts transitively by isometries on $(X,g_{\widetilde{\mathrm{SL}}(2)})$. The following remark shows that the same is true for $\AffR$.

\begin{rem}\label{r:model_SL(2)}
The group $\AffR$ acts simply transitively by isometries on $X$ endowed with the Riemannian metric that corresponds to Thurston's model geometry on $\widetilde{\mathrm{SL}}(2)$.
To see this,
consider the group $G:=\mathrm{Isome}(\widetilde{\mathrm{SL}}(2))$, which has been discussed in \cite[p.464 ff]{doi:10.1112/blms/15.5.401}. It has been show that $G$ consists of two components, say $\Gamma$ and $\Gamma'$. The identity component $\Gamma$ is a $4$-dimensional Lie group generated by the actions of $\R$ and $\widetilde{\mathrm{SL}}(2)$ on $X$. The action of $\widetilde{\mathrm{SL}(2)}$ is immediate, and according to the Iwasawa decomposition, it yields in particular an action of $\Aff$ on $X$.
To explain the action of $\R$, it is useful to see $X$ as a line bundle over $\H^2$. The center of  $\widetilde{\mathrm{SL}}(2)$, which is isomorphic to the additive group $\Z$, acts on $X$ by preserving the line bundle structure and covering the identity map of $\H^2$. This action extends to an action of $\R$ on $X$ by translation of the vertical fibers (this action arises as an action of $S^1$ on $UT(\H^2)$ which covers the identity of $\H^2$ and rotates each fibre by the same angle). Since the action of $\R$ commutes with the action of $\widetilde{\mathrm{SL}}(2)$ (and thus of $\Aff$), we obtain that $\AffR$ acts by isometries on $X$. Moreover, since $\AffR$ acts transitively on the base manifold  $\H^2$  of $X$, and $\R$ acts by translation on the vertical fibers, we see that $\AffR$ acts transitively on $X$. Finally, we argue that the action is free. Assume that $(g,s).x=(g',s').x$ for some $g\in \Aff$, $s\in \R$ and $x\in X$. Then, since the action of $\R$ covers the identity map of $\H^2$, it follows that $g.x$ and $g'.x$ must lie in the same vertical fibre of $X$. As the action of $\Aff$ on $X$ is induced by a free action of $\Aff$ on $\H^2$, it follows that $g=g'$, as desired. This shows that $\AffR$ acts simply transitively by isometries on $(X,g_{\widetilde{\mathrm{SL}}(2)})$.
\end{rem}

%\cite{doi:10.1112/blms/15.5.401}

\begin{rem}
As the classification in Theorem \ref{t:QIClass} shows, already in dimension 3 the property of admitting a lattice (i.e., a discrete subgroup of cofinite volume) is not a quasi-isometric invariant. For example, the group  {$\Aff\times \T^1$} is not unimodular by \cite[Lemma 6.3]{MILNOR1976293} and hence cannot have lattices (see \cite[Section 6]{MILNOR1976293} or \cite[Proposition 2.4.2]{MR2364699}), yet it is quasi-isometrically equivalent to $\mathrm{SL}(2)=\mathrm{SL}(2, \R)$, which    admits the lattice $\mathrm{SL}(2,\Z)$.

\end{rem}

For $k \in \mathbb{N}$, the Iwasawa decomposition of $\mathrm{PSL}(2)_k$ is
\begin{displaymath}
\mathrm{PSL}(2)_k = ANK_k,
\end{displaymath}
where $K_k$ is the $k$-fold cover of the projective special orthogonal group $\mathrm{PSO}(2)$.

Theorem \ref{t:Cowlingetal} applied to  the Iwasawa decomposition of  $\mathrm{PSL}(2)_k$ yields the following result.

\begin{prop}\label{p:PSL(2)_kAffS1}
For every $k\in\mathbb{N}$, the group $\mathrm{PSL}(2)_k$ may be made isometric to  $\Aff \times  \T^1$.
\end{prop}

     \subsection{Bi-Lipschitz groups}

     \begin{prop}\label{p:BL}
     The groups $\mathrm{PSL}(2)_k$ and $\mathrm{PSL}(2)_{k'}$ for different values of $k,k'\in \N$ are bi-Lipschitz equivalent, but cannot be made isometric.
     \end{prop}

The bi-Lipschitz equivalence of $\mathrm{PSL}(2)_k$ and $\mathrm{PSL}(2)_{k'}$ follows easily from Proposition \ref{p:PSL(2)_kAffS1}, but to show that these groups cannot be made isometric, we use \cite[Theorem 2.2]{MR617247} by C.\ Gordon, which we restate here for the reader's convenience.

Assume that $A$ is a connected Lie group with a connected subgroup $G$. Choose Levi factors $G_s$ and $A_s$ of $G$ and $A$, respectively, such that $G_s \subset A_s$, and denote by $\mathfrak{g}_s$ and $\mathfrak{a}_s$ the Lie algebras of $G_s$ and $A_s$. By definition, the Lie algebras $\mathfrak{g}_s$ and $\mathfrak{a}_s$ are semisimple and thus a direct sum of simple Lie algebras, some of which may be compact and others not. This leads to the direct sum decomposition
\begin{displaymath}
\mathfrak{g}_s= \mathfrak{g}_{nc} \oplus \mathfrak{g}_c,
\end{displaymath}
where $\mathfrak{g}_c$ is the direct sum of all compact simple ideals of $\mathfrak{g}_s$ and $\mathfrak{g}_{nc}$ is the direct sum of the remaining simple ideals. In the same way, one decomposes $\mathfrak{a}_s= \mathfrak{a}_{nc} \oplus \mathfrak{a}_c$. By $G_{nc}$ and $A_{nc}$ we denote the connected subgroups of $A$ with Lie algebras $\mathfrak{g}_{nc}$ and $\mathfrak{a}_{nc}$, respectively.

\begin{thm}[Gordon]\label{t:Gordon}
Assume that $A$ is a connected Lie group with a connected subgroup $G$ whose radical is nilpotent. Suppose further that there exists a compact subgroup $K$ of $A$ such that $A= GK$. Then $A_{nc}=G_{nc}$.
\end{thm}

With this theorem at hand, we can prove Proposition \ref{p:BL}.

\begin{proof}[Proof of Proposition \ref{p:BL}]
By Proposition \ref{p:PSL(2)_kAffS1}, both $\mathrm{PSL}(2)_k$ and $\mathrm{PSL}(2)_{k'}$ may be made isometric to $\Aff \times \T^1$. Thus there exist left-invariant Riemannian distances, say $d_k$ and $d_{k'}$ on $\Aff \times \T^1$, as well as $d$ on $\mathrm{PSL}(2)_k$ and $d'$ on $\mathrm{PSL}(2)_{k'}$ such that $(\mathrm{PSL}(2)_k,d)$ is isometric to $( \Aff\times \T^1,d_k)$ and $(\mathrm{PSL}(2)_{k'},d')$ is isometric to $(\Aff\times \T^1,d_{k'})$. Since $d_k$ and $d_{k'}$ are bi-Lipschitz equivalent, it follows that $\mathrm{PSL}(2)_k$ and $\mathrm{PSL}(2)_{k'}$ are bi-Lipschitz equivalent.

Next we show that $\mathrm{PSL}(2)_k$ and $\mathrm{PSL}(2)_{k'}$ cannot be made isometric. For $k\in \N$, we fix a left-invariant Riemannian distance $d_G$ on $G:= \mathrm{PSL}(2)_k$ and we let $A$ be the isometry group of $(G,d_G)$. Then $A= GK$ as in Theorem \ref{t:Gordon}, with $K = \mathrm{Stab}(e)\cap A$, where $\mathrm{Stab}(e)$ denotes the stabilizer of the identity in $G$. Since $G$ is simple, its radical is trivial and hence nilpotent and moreover, $G_{nc}=G$. It follows by  Theorem \ref{t:Gordon} that $G = G_{nc} = A_{nc}$. The same reasoning applies for $k'$ instead of $k$, so that we obtain $G'=A'_{nc}$ for $G'=\mathrm{PSL}(2)_{k'}$ and $A'$ the isometry group of $(G',d_{G'})$. Now if $(G,d_G)$ and $(G',d_{G'})$ were isometric, then $A$ would be isomorphic to $A'$ with an isomorphism given by conjugation via the isometry between $(G,d_G)$ and $(G',d_{G'})$. This would imply that
$\mathrm{PSL}(2)_k = A_{nc}$ is isomorphic to $A'_{nc}= \mathrm{PSL}(2)_{k'}$, which is possible only if $k=k'$ (otherwise the centers of $\mathrm{PSL}(2)_k$ and $\mathrm{PSL}(2)_{k'}$ have different cardinality and hence the groups cannot be isomorphic).
\end{proof}

     \subsection{Quasi-isometrically homeomorphic groups}

     We now consider multiply connected groups that are homeomorphic via a quasi-isometry but not bi-Lipschitz equivalent. The latter fact will be proved by contradiction: if there existed a bi-Lipschitz homeomorphism between the groups it would lift to a bi-Lipschitz homeomorphism of the universal covers according to Proposition \ref{p:lift}. We first recall some basics from covering theory.

     Assume that $G$ is a simply connected Lie group equipped with a left-invariant Riemannian metric $g$. If $N$ is a discrete normal subgroup of $G$, then $G/N$ is a connected Lie group which admits a unique left-invariant Riemannian metric $g_{G/N}$ so that $\pi: (G,g) \to (G/N,g_{G/N})$ becomes a Riemannian covering, that is, a covering map which is locally isometric.

     \begin{prop}\label{p:lift}
     For $i\in \{1,2\}$, let $G_i$ be a simply connected Lie group endowed with a left-invariant Riemannian distance and  let $\pi_i: (G_i,g_i) \to (G_i/N_i,g_{G_i/N_i})$ be a Riemannian covering as above. Then every bi-Lipschitz homeomorphism $f: G_1/N_1 \to G_2/N_2$ lifts to a bi-Lipschitz homeomorphism $\widetilde{f}: G_1 \to G_2$, where `bi-Lipschitz' refers to the Riemannian distances induced by the respective Riemannian metrics.
     \end{prop}

     \begin{proof}
     Let $f: G_1/N_1 \to G_2/N_2$ be bi-Lipschitz. Since $f$ is a homeomorphism and $G_1$ is simply connected, the composition $f\circ \pi_1: G_1 \to G_2/N_2$ is a universal cover of $G_2/N_2$, as is the map $\pi_2: G_2 \to G_2/N_2$. It follows from the uniqueness theorem for universal covers, see for instance \cite[Corollary 13.6]{fulton2013algebraic} or  \cite[I, \Sec 11]{MR1920389}, that there exists a homeomorphism $\widetilde{f}: G_1 \to G_2$ with $\pi_2 \circ \widetilde{f} = f \circ \pi_1$. Since $f$ is bi-Lipschitz and $\pi_1$, $\pi_2$ are local isometries, the map $\widetilde{f}$ is uniformly locally bi-Lipschitz, as is its inverse. Finally, since $G_1$ and $G_2$ are geodesic, $\widetilde{f}$ is bi-Lipschitz as claimed.
     \end{proof}

     \begin{prop}\label{p:QIH}
    Each of the following pairs consists of quasi-isometrically homeomorphic groups that are not bi-Lipschitz equivalent:
     \begin{enumerate}
     \item\label{QIH1} $(\R^2 \times \T^1,\mathrm{N}_3^{\ast}(\R))$
     \item\label{QIH2} $(\mathrm{SE}(2)_k,\mathrm{N}_3^{\ast}(\R))$, for every $k\in \N$.
     \end{enumerate}
     \end{prop}

      \begin{proof}
      Once we know that $\R^2 \times \T^1$ and $\mathrm{N}_3^{\ast}(\R)$ are equivalent via a quasi-isometric but not a bi-Lipschitz homeomorphism, the same statements follow for
      $SE(2)_k$ and $\mathrm{N}_3^{\ast}(\R)$ by Proposition \ref{p:I}, \eqref{I2}. Thus it suffices to prove Part \eqref{QIH1} of Proposition \ref{p:QIH}.

      \medskip

In order to show that
the groups $\mathrm{N}_3^{\ast}(\mathbb{R})$ and $\mathbb{R}^2 \times \mathbb{T}^1$ are quasi-isometric via a homeomorphism,
it is convenient to choose, as we may, coordinates $(x,y,z)$ on $\mathrm{N}_3(\R)$ so that
for all $(x,y,z)$ and $(x',y',z')$, we have
\begin{displaymath}
(x,y,z) \cdot (x',y',z')= (x + x', y + y', z + z' + 2 y  {x'} - 2 x  {y'}).
\end{displaymath}
Without loss of generality we may assume that $\mathrm{N}_3^{\ast}(\mathbb{R})$ is the quotient of
$\mathrm{N}_3(\R)$ by the element
$Z=(0,0,1)$. The Lie group $\mathrm{N}_3(\R) /\langle Z \rangle$ is diffeomorphic to  $\mathbb{R}^2 \times \mathbb{T}^1$.
We see that $\Z^2$ can be identified with a subgroup of the groups $\mathrm{N}_3(\R) /\langle Z \rangle$ and $\mathbb{R}^2 \times \mathbb{T}^1$, respectively, which in both cases acts co-compactly. Moreover, for these particular models, the identity map of $\R^2 \times \mathbb{T}^1$ provides a quasi-isometric homeomorphism between $\mathrm{N}_3^{\ast}(\mathbb{R})$ and $\mathbb{R}^2 \times \mathbb{T}^1$.

Assume towards a contradiction that there exists a biLipschitz map $f: \mathbb{R}^2 \times \mathbb{T}^1 \to \mathrm{N}_3^{\ast}(\mathbb{R})$.
It  follows from Proposition \ref{p:lift} that there would exists a bi-Lipschitz homeomorphism $\widetilde{f}: \R^3 \to \mathrm{N}_3(\mathbb{R})$.
This is known to be false, for instance because $\R^3$ has volume growth of order $3$, whereas the volume of balls in $\mathrm{N}_3(\mathbb{R})$ grows with order $4$ at large. We have thus proven that $\mathrm{N}_3^{\ast}(\mathbb{R})$ is not bi-Lipschitz equivalent to
$\mathbb{R}^2 \times \mathbb{T}^1$.
\end{proof}

     \subsection{Quasi-isometric groups}\label{s:QI}

     \begin{prop}\label{p:QI}
     Each of the following pairs consists of quasi-isometrically equivalent groups that are not equivalent via a quasi-isometric homeomorphism:
     \begin{enumerate}
     \item\label{QI1} $(G,H)$ for distinct $G,H \in \{\T^1, \T^2, \T^3, \mathrm{SU}(2), \mathrm{SO}(3)\}$
     \item\label{QI2} $(G,H)$ for distinct $G,H \in \{\R, \R\times \T^1, \R\times \T^2\}$
     \item\label{QI3} $(\R^2,\R^2 \times \T^1)$
     \item\label{QI4} $(\Aff, \Aff\times \T^1)$
     \item\label{QI5} $(\R^2,\mathrm{N}_3^{\ast}(\R))$
      \item\label{QI6} $(\R^2,\mathrm{SE}(2)_k)$, for every $k\in \N$
     \item\label{QI7} $(\Aff,\mathrm{PSL}(2)_k)$, for every $k\in \N$.
     \end{enumerate}
     \end{prop}

      \begin{proof}

       The groups appearing on the same line in Proposition \ref{p:QI} are topologically distinct and hence cannot be equivalent via a quasi-isometric homeomorphism. Indeed, denoting by ``$\simeq$'' equivalence via a diffeomorphism of manifolds, we have:
       \begin{enumerate}
     \item $\T^1\simeq \S^1$, $\T^2 \simeq \S^1 \times \S^1$, $\T^3 \simeq \S^1 \times \S^1 \times \S^1$, $\mathrm{SU}(2)\simeq \S^3$, $\mathrm{SO}(3)\simeq P\R^3$
     \item $\R$, $\R\times \T^1\simeq \R \times \S^1$, $\R\times \T^2\simeq \R \times \S^1\times \S^1$
     \item $\R^2$ and $\R^2 \times \T^1\simeq \R^2 \times \S^1$
     \item $\Aff\simeq \R^2$ and $\Aff\times \T^1\simeq \R^2 \times \S^1$
     \item $\R^2$ and $\mathrm{N}_3^{\ast}(\R)\simeq \R^2 \times \S^1$
      \item $\R^2$ and $\mathrm{SE}(2)_k\simeq \R^2 \times \S^1$
     \item $\Aff\simeq \R^2$ and $\mathrm{PSL}(2)_k\simeq \R^2 \times \S^1$.
     \end{enumerate}

     \medskip

     It remains to show that groups appearing on the same line are quasi-isometrically equivalent, even if they are not homeomorphic. First, the groups $\mathbb{T}^1$, $\mathbb{T}^2$, $\mathbb{T}^3$, $\mathrm{SU}(2)$, and  $\mathrm{SO}(3)$ are trivially quasi-iso{\-}metri{\-}cally equivalent because they are compact.

Second, the groups $\mathbb{R}$, $\mathbb{R}\times \mathbb{T}^1$,  and $\mathbb{R}\times \mathbb{T}^2$ are clearly quasi-isometrically equivalent. More generally,  $\mathbb{R}\times K$ is quasi-isometric to $\mathbb{R}\times K'$ for arbitrary compact Lie groups $K$ and $K'$, as one can see by arguing componentwise.
For the same reason, $\mathbb{R}^2$ and $\mathbb{R}^2 \times \mathbb{T}^1$ are quasi-isometrically equivalent, and so are $\Aff$ and $ \Aff\times \T^1$.
Having established the quasi-isometric equivalence in the cases \eqref{QI1} -- \eqref{QI4}, the remaining cases follow by transitivity. Indeed, the information from Propositions~\ref{p:QIH}, \ref{p:I}, and \ref{p:PSL(2)_kAffS1} can be used to deduce that the groups
      in \eqref{QI5}, \eqref{QI6}, and \eqref{QI7} are quasi-isometrically equivalent, once this has been established for the groups in \eqref{QI3} and \eqref{QI4}.

     \end{proof}

\section{Conclusion of the quasi-isometric classification}\label{s:conclusion}

In Section \ref{s:relations} we have identified pairs of Riemannian Lie groups that are quasi-isometrically equivalent. In this section we show that all remaining pairs of at most three-dimensional connected Lie groups are quasi-isometrically distinct, thus establishing Theorem \ref{t:QIClass}. The proof uses the following quasi-isometric invariants of connected Riemannian Lie groups:

\begin{itemize}
\item \emph{degree of polynomial volume growth}
\item \emph{polynomial volume growth} (or equivalently by \cite{MR0369608,MR0316625}: type R)
\item \emph{Gromov hyperbolicity} \cite{MR919829}, see also e.g.\ \cite[Theorem 3.1.11]{Pajot2017}
\item \emph{covering dimension of asymptotic cones} \cite{MR2399134}.
\end{itemize}

Besides these general quasi-invariants, we also rely on quasi-isometric classification results for connected Riemannian Lie groups of a specific form:

\begin{itemize}
\item for Gromov hyperbolic connected Riemannian Lie groups (which are proper metric spaces): \emph{topology of the boundary} \cite{MR919829}, see also e.g.\ \cite[Proposition 2.20]{MR1921706}
\item for simply connected Riemannian manifolds of negative or zero curvature: \emph{$L^p$ cohomology} \cite{MR1253544}
\item \cite[Corollaire 1]{Pansu} and \cite[Corollary 1.3]{MR3180486} for $\R^n \rtimes_A \R$ with $A\in\mathrm{Mat}(n\times n)$ having only eigenvalues with positive real parts\\
(in our notation this applies to: $J$, $D_{\lambda}$ for $0<\lambda\leq 1$, $C_{\lambda}$ for $\lambda>0$)
\item \cite[Theorem 1.3]{MR2925383} for $\mathrm{Sol}(m,n)$, the solvable Lie groups $\R^2  \rtimes \R$, where $\R$ acts by $z\cdot (x,y)= (e^{mz}x,e^{-nz}y)$, for  {$m>n>0$} using \emph{coarse differentiation}\\
(in our notation: $\mathrm{Sol}(1,-\lambda)=D_{\lambda}$  {for $-1<\lambda<0$})
\end{itemize}

 \begin{proof}[Proof of Theorem \ref{t:QIClass}]
We first discuss why the listed classes are quasi-iso{\-}metri{\-}cally distinct.

\medskip

\emph{The groups in classes (1) -- (5)} are the only groups of type R, as can be seem from an explicit description of the Bianchi classification of Lie algebras, as given for instance in \cite[Chapter 7, \Sec 1.1]{gorbatsevich1994lie}.
The individual classes are divided according to the degree $d\in \{0,1,2,3,4\}$ of polynomial volume growth.

\medskip

\emph{The groups in classes (6) and $(7_{\lambda})$} have exponential growth but are not Gromov hyperbolic: for the groups in class (6) this is easy to see since $\AffR$ can be endowed with a left-invariant Riemannian metric such that it contains an isometrically embedded copy of $\R^2$. The proof that $D_{\lambda}$ is not hyperbolic for $\lambda <0$ is given below in Proposition \ref{p:L(3,2,-1<eta<0)}.

We now show that (6) and $(7_{\lambda})$ are distinct classes. The group $\AffR$ is not quasi-iso{\-}me{\-}tri{\-}cally equivalent to any $D_{\lambda}$ since the covering dimension of the asymptotic cone of $D_{\lambda}$ is  $1$ for every $\lambda$, while $\AffR$ has cone dimension $2$ by \cite[Theorem 1.1]{MR2399134}.

To distinguish the classes $(7_{\lambda})$  for different values of $\lambda \in [-1,0)$, take $-1\leq \lambda_1 < \lambda_2<0$. If $\lambda_1\neq -1$, then $D_{\lambda_1}$ is not quasi-isometric to $D_{\lambda_2}$ by \cite[Theorem 1.3]{MR2925383}. If $\lambda_1=-1$, then $D_{\lambda_1}= D_{-1}$ is the Lie group of the  Solv geometry, which by \cite[Section 2]{Lee2008} and \cite[Section 3]{Brady01072001} admits a cocompact lattice of the form $\mathbb{Z} \ltimes \mathbb{Z}^2$, while there does not exist any finitely generated group quasi-isometric to $D_{\lambda_2}$ by \cite[Theorem 1.2]{MR2925383}.

\medskip

\emph{The groups in classes (8) -- $(11_{\lambda})$} are Gromov hyperbolic: since $\Aff$, $J$, and $D_\lambda$ for $\lambda \in (0,1]$ are all of the form $\R^n \rtimes_A \R$ for a matrix $A$ whose eigenvalues all have positive real parts, it follows from \cite[Theorem 3]{MR0353210}  that each of these groups admits a left-invariant Riemannian metric with negative sectional curvature bounded away from zero. Finally, a simply connected complete Riemannian manifold with negative curvature bounded away from zero is Gromov hyperbolic, see for instance \cite[p.52, Corollaire 10]{MR1086648}. While the groups in (8) have $\S^1$ as visual boundary, the groups in (9)-$(11_{\lambda})$ have $\S^2$.

All groups $J$, $D_{\lambda}$ ($\lambda \in (0,1]$)  are of the form $\mathbb{R}^2 \rtimes_A \mathbb{R}$ with $A$  {equal to} $\begin{psmallmatrix}1&1\\0&1\end{psmallmatrix}$ or $\begin{psmallmatrix}1&0\\ {0}&\lambda\end{psmallmatrix}$, $\lambda \in (0,1]$.
It is a special case of  \cite[Corollaire 1]{Pansu}, proved by means of $L^p$ cohomology, that two groups in the family $D_{\lambda}$, $D_{\lambda'}$, $\lambda,\lambda' \in (0,1]$ are quasi-isometrically equivalent if and only if they are isomorphic, that is, if and only if $\lambda = \lambda'$.
The quasi-isometric classification of all negatively curved $\R^n \rtimes \R$ has been completed in \cite{MR3180486}.
As a special case of \cite[Corollary 1.3]{MR3180486}, if $A$ and $B$ are $2\times 2$ matrices whose eigenvalues all have positive real parts, then the two groups $\mathbb{R}^2 \rtimes_A \mathbb{R}$ and $\mathbb{R}^2 \rtimes_B \mathbb{R}$ are quasi-isometric if and only if there exists $s>0$ such that  {$A$ and $sB$ have same real part Jordan form. This shows in particular that $J$ cannot be quasi-isometric to any $D_{\lambda}$, $\lambda \in (0,1]$, and $D_{\lambda}$ is quasi-isometric to $D_{\lambda'}$ only if $\lambda = \lambda'$.}

\medskip

Except for $(7_{\lambda})$ and $(11_{\lambda})$, which represent uncountably many different classes, all the groups listed on one line in Table \ref{t:QIclass} are quasi-isometrically equivalent: this follows from Propositions \ref{p:I}, \ref{p:BL}, \ref{p:QIH}, and \ref{p:QI}.
\end{proof}

We now discuss the proof of one result which has been used in the quasi-isometric classification.

\begin{prop}\label{p:L(3,2,-1<eta<0)}
The Lie groups $D_{\lambda}$, $\lambda \in [-1,0)$, are not Gromov hyperbolic.
\end{prop}

There are different proofs available for this fact. One can show for instance that the Dehn function of $D_{\lambda}$, $\lambda \in [-1,0)$,
is exponential (the argument for $D_{-1}$ is outlined in \cite{MR3034292}), and then use a result by Gromov
\cite{MR919829} to deduce that $D_{\lambda}$, $\lambda \in [-1,0)$ is not Gromov-hyperbolic since the Dehn function is not linear.
 Another possibility would be to consider the asymptotic cone of $D_{\lambda}$, $\lambda \in [-1,0)$; see
 \cite{MR1709665} and references therein.
A proof for Proposition \ref{p:L(3,2,-1<eta<0)} is also  contained in  \cite[\Sec 3.1]{MR2925383},  {where it was observed that points in $D_{\lambda}$, $\lambda \in [-1,0)$,  which are not contained in the same hyperbolic plane can be joined by quasi-geodesics that do not lie close to each other. We recall the argument below.} It is convenient to think of the hyperbolic plane $\mathbb{H}^2$ not as the upper half plane $\{(u,v):\;v>0\}$ with the metric given by $$\mathrm{d}s^2=\frac{1}{v^2}(\mathrm{d}u^2 + \mathrm{d}v^2),$$ but rather to apply a coordinate transform $(x, {z})=(u,\log v)$. Then $\mathbb{H}^2$ can be seen as $\mathbb{R}^2$ equipped with the metric given by $\mathrm{d}s^2 = e^{-2z}\;\mathrm{d}x^2 + \mathrm{d}z^2$. It turns out that the groups $D_{\lambda}$, $\lambda \in [-1,0)$, are all foliated by isometrically embedded copies of $\mathbb{H}^2$. Perpendicular to these planes, there is another family of homothetically embedded `upside down' versions of $\mathbb{H}^2$.

\begin{proof}[Proof of Proposition \ref{p:L(3,2,-1<eta<0)} following \cite{MR2925383}]
Recall that  $D_{\lambda}$  is $\R^3$ with the group law
\begin{displaymath}
(x,y,z)\ast (x',y',z')=(x+e^z x', y + e^{\lambda z}y',z+z').
\end{displaymath}
Let $g_{\lambda}$ be the metric on $D_{\lambda}$ which makes the following left-invariant frame orthonormal:
\begin{displaymath}
E_1=e^z\partial_x,\quad E_2=e^{\lambda z}\partial_y,\quad E_3= -\partial_z.
\end{displaymath}
(Note that $\{E_1,E_2,E_3\}$ has structure constants as described in \eqref{eq:structure_constants}.)
The associated length element is given by
\begin{displaymath}
\mathrm{d}s^2 = e^{-2z}\;\mathrm{d}x^2 + e^{-2 \lambda z}\;\mathrm{d}y^2 + \mathrm{d}z^2.
\end{displaymath}
It follows that the planes $\{y=\mathrm{const}\}$ are isometrically embedded copies of $\H^2$, whereas the planes $\{x=\mathrm{const}\}$ are homothetically embedded copies of the reflected hyperbolic plane.

Consider two points $p_1=(x_1,y_1,z_1)$ and $p_2=(x_2,y_2,z_2)$ in $D_{\lambda}$ with $x_1\neq x_2$ and $y_1\neq y_2$. We will construct two quasi-geodesics $\gamma_a$ and $\gamma_b$ which connect $p_1$ and $p_2$ but do not lie close to each other. First, we let $\gamma_{a,1}$ be the geodesic segment between $p_1$ and $(x_2,y_1,z_2)$ inside the hyperbolic plane $\{y=y_1\}$. Then we let $\gamma_{a,2}$ be the geodesic segment in $\{x=x_2\}$ connecting the endpoint of $\gamma_{a,1}$ to $p_2$, and we denote by $\gamma_a$ the concatenation of $\gamma_{a,1}$ and $\gamma_{a,2}$. The curve $\gamma_b$ is obtained in an analogous way, by first connecting $p_1$ to $(x_1,y_2,z_2)$ by a geodesic segment in the plane $\{x=x_1\}$, and then connecting the point $(x_1,y_2,z_2)$ to $p_2$ by a geodesic in the hyperbolic plane $\{y=y_2\}$. Observe that the map
\begin{displaymath}
D_{\lambda} \to \H^2 \times \H^2,\quad (x,y,z)\mapsto ((x,z),(y,z))
\end{displaymath}
is a quasi-isometric embedding with constants depending only on the parameter $\lambda$ if $D_{\lambda}$ is endowed with the distance induced by $g_{\lambda}$ and $\H^2 \times \H^2$ is equipped with a product metric of $d_{\H^2}$, where $d_{\H^2}$ is induced by a metric of sectional curvature constant equal to $-1$. It follows that both $\gamma_a$ and $\gamma_b$ are $(L,C)$-quasi-geodesics, for constants $L=L(\lambda)\geq 1$ and $C=C(\lambda) \geq 0$ independent of $a$ and $b$. By applying this construction to a sequence of points $p_{1,n}=(x_{1,n},y_{1,n},z)$ and $p_{2,n}=(x_{2,n},y_{2,n},z)$, with $z\in \R$, $|x_{1,n}-x_{2,n}|\to \infty$ and  $|y_{1,n}-y_{2,n}|\to \infty$ as $n\to \infty$, we see that there does not exist a constant $\delta>0$ such that for every $n$, the curve $\gamma_{a}$ connecting $p_{1,n}$ to $p_{2,n}$ is contained in the $\delta$-neighborhood of $\gamma_b$. This proves that $(D_{\lambda},g_{{\lambda}})$ is not Gromov hyperbolic.
\end{proof}

\bibliographystyle{acm}
\bibliography{references}

\end{document}